\documentclass[a4paper,12pt]{article}

\makeatletter
\@addtoreset{footnote}{page}
\makeatother

\usepackage{amsthm}
\usepackage{amsmath,amssymb,latexsym,amsfonts,mathrsfs,xcolor}

\renewcommand{\hat}{\widehat}

\theoremstyle{plain}
\newtheorem{Thm}{Theorem}[section]

\newtheorem{Lem}[Thm]{Lemma}

\theoremstyle{definition}

\numberwithin{equation}{section}

\makeatletter
\renewcommand\section{\@startsection {section}{1}{\z@}%
                                   {-3.5ex \@plus -1ex \@minus -.2ex}%
                                   {2.3ex \@plus.2ex}%
                                   {\normalfont\large\bf}}
\makeatother

\makeatletter
\renewcommand\subsection{\@startsection {subsection}{1}{\z@}%
                                   {-3.5ex \@plus -1ex \@minus -.2ex}%
                                   {2.3ex \@plus.2ex}%
                                   {\normalfont\normalsize\bf}}
\makeatother

\allowdisplaybreaks

\begin{document}

\begin{center}
{\Large \bf 
Limit theorems for the fluctuation of the dynamic elephant random walk in the superdiffusive case
}
\end{center}
\begin{center}
Go Tokumitsu$^{1}$\quad and\quad Kouji Yano$^{2}$

Department of Mathematics, Graduate School of Science, Osaka University, Toyonaka, Osaka, 560-0043, Japan$^{1,2}$

(Electronic mail: go.tokumitsu@gmail.com$^{1}$, kyanomath@gmail.com$^{2}$)
\end{center}

\begin{abstract}
Motivated by the previous results by Coletti--de Lima--Gava--Luiz (2020) and Shiozawa (2022), we study the fluctuation of the dynamic elephant random walk in the superdiffusive case with a strong elephant component. Applying the martingale convergence theorem, we prove the Central Limit Theorem and the Law of Iterated Logarithm, where a random drift is subtracted from the process considered. 
\end{abstract}
\section{Introduction}
The Elephant Random Walk (ERW), which was introduced by Sch\"utz and Trimper \cite{Sch} in 2004, is a one-dimensional discrete-time stochastic process with infinite memory, and is a simple model exhibiting anomalous diffusion. The Dynamic Random Walk (DRW), which was introduced by Guillotin-Plantard \cite{Gui} in 2000, is a one-dimensional discrete-time random walk whose transition probabilities are determined by the orbit of a discrete-time dynamical system; see \cite{Gui2} for the details. Coletti--de Lima--Gava--Luiz \cite{Col} in 2020 introduced the Dynamic Elephant Random Walk (DERW) which is a random mix of the ERW and the DRW. They showed that the DERW exhibits two distinct phase transitions.

The Central Limit Theorem (CLT) and the Law of the Iterated Logarithm (LIL) for the DERW have been studied by Kubota--Takei \cite{KaT}, Coletti--de Lima--Gava--Luiz \cite{Col} and Shiozawa \cite{Sio}. In Coletti et al. \cite{Col}, the CLT and the LIL are proved for the DERW in the superdiffusive case with a weak elephant component. In Kubota--Takei \cite{KaT} and Shiozawa \cite{Sio}, under the assumption of coefficient convergence, they prove the CLT and the LIL in the superdiffusive case with a strong elephant component, where a random drift is subtracted from the process considered. In this paper, we prove the CLT and the LIL for the DERW in the superdiffusive case with a strong elephant component, without assuming the  coefficient convergence.

\subsection{The definition of the dynamic elephant random walk}
Let us introduce the DERW in an equivalent formulation. Let $p,q\in[0,1]$ be constants and let $\alpha=\{\alpha_n\}_{n=1}^\infty$ and $\beta=\{\beta_n\}_{n=1}^\infty$ be sequences of $[0,1]$. The DERW is a stochastic process $\{S_n\}_{n=0}^\infty$ on a probability space $(\Omega,\mathcal{F},P)$ which takes values in $\mathbb{Z}$ such that $S_0:=0$ and the increment $X_n:=S_n-S_{n-1}$ for $n\ge1$ takes values in $\{-1,+1\}$ and satisfies
\begin{align}\label{CE1}
E[X_1]=\alpha_1(2q-1)+(1-\alpha_1)(2\beta_1-1)
\end{align}
and
\begin{align}\label{CE}
E[X_{n+1}|\mathcal{F}_n^X]=\frac{\alpha_{n+1}(2p-1)}{n}\cdot S_n+(1-\alpha_{n+1})(2\beta_{n+1}-1)\quad (n\ge1).
\end{align}
Here $\mathcal{F}_n^X:=\sigma(X_1,\ldots,X_n)$ denotes the natural filtration generated by $\{X_n\}_{n=1}^\infty$. The recursive equations \eqref{CE1} and \eqref{CE} determine the joint distribution of $\{S_n\}_{n=0}^\infty$. In fact, since $S_n=\sum_{k=1}^nX_k$, the joint distribution of $\{S_n\}_{n=0}^\infty$ is determined by that of $\{X_n\}_{n=1}^\infty$, whose finite dimensional distributions are determined by
\begin{align*}
P(X_1=\pm1)=\frac{1\pm E[X_1]}{2}
\end{align*}
and
\begin{align*}
P(X_{n+1}=\pm1|\mathcal{F}_{n}^X)=\frac{1\pm E[X_{n+1}|\mathcal{F}_n^X]}{2}\quad(n\ge1).
\end{align*}
The DERW can be concretely constructed in the following way. Let $X_1$, $D_1,D_2,\ldots$, $C_1,C_2,\ldots$, $U_1,U_2,\ldots$, $\delta_1,\delta_2,\ldots$ be independent random variables whose distributions are given by
\begin{align*}
P(X_1=\pm1)&=\frac{1}{2}\pm\frac{1}{2}(\alpha_1(2q-1)+(1-\alpha_1)(2\beta_1-1))\quad(n\ge0),\\ U_{n+1} \ &\text{is uniform on $\{1,\ldots,n+1\}$},\\
D_{n+1}&:=\begin{cases}
1&\text{with probability $\beta_{n+1}$}\\
-1&\text{with probability $1-\beta_{n+1}$}
\end{cases},\\
C_{n+1}&:=\begin{cases}
1&\text{with probability $p$}\\
-1&\text{with probability $1-p$}
\end{cases},\\
\delta_{n+1}&:=\begin{cases}
1&\text{with probability $\alpha_{n+1}$}\\
0&\text{with probability $1-\alpha_{n+1}$}
\end{cases}.
\end{align*}
We now define
\begin{align*}
X_{n+1}:=\begin{cases}
C_{n+1}X_{U_{n+1}}&(\delta_{n+1}=1)\\
D_{n+1}&(\delta_{n+1}=0)
\end{cases}.
\end{align*}
Set $\mathcal{F}_n:=\sigma(X_1,D_j,C_j,U_j,\delta_j;j\le n)$. Then we obtain, for $n\ge0$,
\begin{align*}
E[X_{n+1}|\mathcal{F}_n]&=\alpha_{n+1}E[C_{n+1}X_{U_{n+1}}|\mathcal{F}_n]+(1-\alpha_{n+1})E[D_{n+1}|\mathcal{F}_n]\\
&=\alpha_{n+1}\cdot\frac{1}{n}\sum_{k=1}^nX_k\cdot E[C_{n+1}]+(1-\alpha_{n+1})(2\beta_{n+1}-1)\\
&=\frac{\alpha_{n+1}(2p-1)}{n}\cdot S_n+(1-\alpha_{n+1})(2\beta_{n+1}-1),
\end{align*}
which shows  \eqref{CE} since $\mathcal{F}_n^{X}\subset \mathcal{F}_n$.

Let us check that our formulation is equivalent to that of the dynamic elephant random walk (DERW) introduced in Coletti et al. \cite{Col}. Let $(\mathcal{X},\mathscr{A},\mu,T)$ be a measure-preserving dynamical system, that is, the triplet $(\mathcal{X},\mathscr{A},\mu)$ is a probability space and $T:\mathcal{X}\to\mathcal{X}$ is a $\mu$-invariant transformation. Let $p,q\in[0,1]$ be constants and let $f:\mathcal{X}\to[0,1]$ be a measurable mapping. If we take $\beta_n=f(T^nx)$, then the recursive equations \eqref{CE1} and \eqref{CE} become
\begin{align*}
E[X_1]=\alpha_1(2q-1)+(1-\alpha_1)(2f(Tx)-1)
\end{align*}
and
\begin{align*}
E[X_{n+1}|\mathcal{F}_n^X]=\frac{\alpha_{n+1}(2p-1)}{n}\cdot S_n+(1-\alpha_{n+1})(2f(T^{n+1}x)-1)\quad(n\ge1),
\end{align*} 
which shows that our formulation reduces to that of Coletti et al. \cite{Col}. Next, we start with our formulation \eqref{CE1} and \eqref{CE}. Let us construct a measure-preserving dynamical system as follows: Let $\mathcal{X}=[0,1]^{\mathbb{N}_0}$ denote the space of infinite sequences taking values in $[0,1]$, where $\mathbb{N}_0=\{0,1,2,\ldots\}$. The space $\mathcal{X}$ is equipped with the Borel $\sigma$-algebra $\mathscr{A}$ and the product measure $\mu$ on $\mathcal{X}$ of the Lebesgue measure on $[0,1]$. The obtained triplet $(\mathcal{X},\mathscr{A},\mu)$ is a probability space. We define the shift operator $T:\mathcal{X}\to \mathcal{X}$ by
\[T(\beta_0,\beta_1,\beta_2,\ldots)=(\beta_1,\beta_2,\beta_3,\ldots),\]
which is a measure-preserving transformation of $(\mathcal{X},\mathscr{A},\mu)$. Thus, the obtained quadruple $(\mathcal{X},\mathscr{A},\mu,T)$ is a measure-preserving dynamical system. Take $f: \mathcal{X}\to[0,1]$ as the projection
\[f(\beta_0,\beta_1,\ldots)=\beta_0\]
and set $\beta_n=f(T^n\beta)$ with $\beta=(\beta_0,\beta_1,\ldots)\in\mathcal{X}$ for all $n\ge1$. Then the formulation of Coletti et al. \cite{Col} reduces to our formulation.
\subsection{Superdiffusive case}
We confine ourselves to the case $p>3/4$, which is called the \emph{superdiffusive case}. We will write for short 
\begin{align*}
\ell_{\inf}(\alpha):=\liminf_{n\to\infty}\alpha_n,\quad \ell_{\sup}(\alpha):=\limsup_{n\to\infty}\alpha_n,
\end{align*}
and
\begin{align*}
\ell_{\inf}(\beta):=\liminf_{n\to\infty}\beta_n,\quad \ell_{\sup}(\beta):=\limsup_{n\to\infty}\beta_n.
\end{align*}
Let $a_1:=1$ and
\begin{align*}
a_n:=\prod_{k=1}^{n-1}\left(1+\frac{(2p-1)\alpha_{k+1}}{k}\right)\quad \text{for}\ n\ge2.
\end{align*}

Set
\begin{align*}
A_n^2:=\sum_{k=1}^n\frac{1}{a_k^2}(1-E[X_k]^2),\quad B_n^2:=\sum_{k=1}^n\frac{1}{a_k^2}\quad \text{for $n\in\mathbb{N}\cup\{\infty\}$}.
\end{align*}
Coletti et al. \cite{Col} have obtained the following theorem concerning the CLT and LIL.
\begin{Thm}[Coletti et al. {\cite[Theorem 3, Theorem 6 and Corollary 4]{Col}}]
Let $\{S_n\}_{n=0}^\infty$ be the DERW. Suppose one of the following two conditions:
\begin{enumerate}
\item[(I)] $3/4<p<1$, $0<\ell_{\inf}(\alpha)\le\ell_{\sup}(\alpha)<\frac{1}{4p-2}$.
\item[(II)] $p=1$ and $0<\ell_{\inf}(\alpha)\le\ell_{\sup}(\alpha)<1$, $0<\ell_{\inf}(\beta)\le\ell_{\sup}(\beta)<\frac{1-\ell_{\sup}(\alpha)}{1-\ell_{\inf}(\alpha)}$.
\end{enumerate}
Then the following assertions hold:
\begin{enumerate}
\item (Central Limit Theorem)
\begin{align}\label{CLT}
\frac{S_n-E[S_n]}{a_nA_n}\overset{d}{\to}N(0,1),
\end{align}
\item (Law of Iterated Logarithm)
\end{enumerate}
\begin{align*}
\limsup_{n\to\infty}\pm\frac{S_n-E[S_n]}{a_nA_n\sqrt{2\log\log(A_n)}}=1\quad \text{a.s.}
\end{align*}
\end{Thm}
By a \emph{weak elephant component} we mean the condition $\ell_{\sup}(\alpha)<\frac{1}{4p-2}$. In addition to the phase transition at $p=3/4$, Coletti et al. \cite{Col} obtained another phase transition as the following theorem with a \emph{strong elephant component} in the sense that $\ell_{\inf}(\alpha)>\frac{1}{4p-2}$, where the CLT \eqref{CLT} breaks down.
\begin{Thm}[Coletti et al. {\cite[Theorem 7]{Col}}]\label{nondegThm}
Let $\{S_n\}_{n=0}^\infty$ be the DERW with $p>3/4$ and $\ell_{\inf}(\alpha)>\frac{1}{4p-2}$. Then
\begin{align}\label{nondeg}
\frac{S_n-E[S_n]}{a_n}\to M\quad \text{a.s. and in $L^2$},
\end{align} 
where $M$ is a non-degenerate zero mean random variable.
\end{Thm}
Note that, by the proof of Theorem 7 of \cite{Col}, Theorem \ref{nondegThm} holds even if the condition that $p>3/4$ and $\ell_{\inf}(\alpha)>\frac{1}{4p-2}$ is replaced by the assumption $B_\infty^2<\infty$. Let us state our main theorem. We obtain the fluctuation limits of \eqref{nondeg} of Theorem \ref{nondegThm}.
\begin{Thm}\label{mThm}
Let $\{S_n\}_{n=0}^\infty$ be the DERW. Suppose one of the following two conditions:
\begin{enumerate}
\item[(I)] $3/4<p<1$ and $\ell_{\inf}(\alpha)>\frac{1}{4p-2}.$
\item[(II)] $p=1$ and $\frac{1}{4p-2}<\ell_{\inf}(\alpha)\le\ell_{\sup}(\alpha)<1$, $0<\ell_{\inf}(\beta)\le\ell_{\sup}(\beta)<\frac{1-\ell_{\sup}(\alpha)}{1-\ell_{\inf}(\alpha)}$.
\end{enumerate} 
Then the following assertions hold:
\begin{enumerate}
\item (\text{Central Limit Theorem})
\begin{align*}
\frac{S_n-E[S_n]-a_nM}{a_n\sqrt{A_\infty^2-A_n^2}}\overset{d}{\to}N(0,1).
\end{align*}
\item (\text{Law of Iterated Logarithm})
\begin{align*}
\limsup_{n\to\infty}\pm\frac{S_n-E[S_n]-a_nM}{a_n\sqrt{2(A_\infty^2-A_n^2)\log|\log(A_\infty^2-A_n^2)|}}=1\quad \text{a.s.}
\end{align*}
\end{enumerate}
\end{Thm}
If $p=1$ and $\varepsilon_n=2\beta_n-1$ with $\beta_n\in[1/2,1]$, then Theorem \ref{mThm} generalizes some results of Shiozawa \cite{Sio}, which deals with the case $\alpha_n$ and $\beta_n$ converge. In particular, Theorem \ref{mThm} also generalizes some results of Kubota--Takei \cite{KaT}.

\subsection{Other previous results}
Note that the DERW generalizes the elephant random walk (ERW); see \cite{Ber},\cite{Col2},\cite{Col3},\cite{KaT},\cite{Sch} for the details. In fact, if we take $\alpha_n\equiv1$, then the recursive equations \eqref{CE1} and \eqref{CE} become
\begin{align*}
E[X_1]=2q-1
\end{align*}
and
\begin{align*}
E[X_{n+1}|\mathcal{F}_n^X]=\frac{2p-1}{n}\cdot S_n\quad(n\ge1),
\end{align*}
which shows that DERW becomes the ERW. 

Note also that the DERW generalizes Kubota-Takei \cite{KaT}. In fact, if we take $p=1$ and $\alpha_1=1$, $\alpha_n\equiv\alpha\in[0,1]$ for all $n\ge2$, $\varepsilon_n=2\beta_n-1$ with $\beta_n\in[1/2,1]$, then the recursive equations \eqref{CE1} and \eqref{CE} become
\begin{align*}
E[X_1]=2q-1
\end{align*}
and
\begin{align*}
E[X_{n+1}|\mathcal{F}_n^X]=\frac{\alpha}{n}S_n+(1-\alpha)\varepsilon_{n+1}\quad(n\ge1),
\end{align*} 
which shows that the DERW becomes Kubota--Takei \cite{KaT}. 

Note also that the DERW generalizes Shiozawa's generalized elephant random walk; see \cite{Sio} for the details. In fact, if we take $p=1$ and $\alpha_1=1$, $\varepsilon_n=2\beta_n-1$ with $\beta_n\in[1/2,1]$, then the recursive equations \eqref{CE1} and \eqref{CE} become
\begin{align*}
E[X_1]=2q-1
\end{align*}
and
\begin{align*}
E[X_{n+1}|\mathcal{F}_n^X]=\frac{\alpha_{n+1}}{n}S_n+(1-\alpha_{n+1})\varepsilon_{n+1}\quad(n\ge1),
\end{align*} 
which shows that the DERW becomes Shiozawa's generalized elephant random walk. 

Let us recall fluctuation limit theorems which have been obtained for the Kubota--Takei \cite{KaT}. 
The following deals with the superdiffusive case.
\begin{Thm}[Kubota--Takei {\cite[Theorem 3]{KaT}}]\label{KaTThm}\ Let $\varepsilon_n\to\varepsilon\in[0,1)$. If $3/4< p< 1$, then the following assertions hold:
\begin{enumerate}
\item (\text{Central Limit Theorem})
\begin{align}\label{KaT1}
\frac{S_n-E[S_n]-a_nM}{\sqrt{n}}\overset{d}{\to}N\left(0,\frac{1-\varepsilon^2}{4p-3}\right).
\end{align}
\item (\text{Law of Iterated Logarithm})
\begin{align}\label{KaT2}
\limsup_{n\to\infty}\pm\frac{S_n-E[S_n]-a_nM}{a_n\cdot\hat{\phi}\left({\frac{1-\varepsilon^2}{4p-3}n}\right)}=1\quad \text{a.s.}
\end{align}
\end{enumerate}
where $\hat{\phi}(t):=\sqrt{2t\log|\log t|}.$
\end{Thm}
Let us recall fluctuation limit theorems which have been obtained for Shiozawa's generalized elephant random walk.
\begin{Thm}[Shiozawa {\cite[Theorem 22]{Sio}}]\label{ShiomThm}
Let $\varepsilon_n\to\varepsilon\in[0,1)$ and $\alpha_n\to\alpha\in[0,1)$, $\sum_{n=1}^\infty1/a_n^2<\infty$. If $\alpha_n\to1$, then assume also that $a_n/n\to0$ as $n\to\infty$. Then the following assertions hold:
\begin{enumerate}
\item[(i)] (\text{Central Limit Theorem})
\begin{align*}\label{}
\frac{S_n-E[S_n]-a_nM}{a_n\sqrt{B_\infty^2-B_n^2}}\overset{d}{\to}N(0,1-\varepsilon^2).
\end{align*}
\item[(ii)] (\text{Law of Iterated Logarithm})
\begin{align*}\label{}
\limsup_{n\to\infty}\pm\frac{S_n-E[S_n]-a_nM}{a_n\hat{\phi}(B_\infty^2-B_n^2)}=\sqrt{1-\varepsilon^2}\quad \text{a.s.}
\end{align*}
\end{enumerate}
\end{Thm}
\subsection{Organization of this paper}
This paper is organized as follows. In Section 2, we study summability and asymptotics for sequences related to $\{a_n\}_{n=1}^\infty$. In Section 3, we give the proof of Theorem \ref{mThm}.
\section{Summability and asymptotics for sequences related to $\{a_n\}_{n=1}^\infty$}
We recall the following lemma without proof concerning the finiteness of $B_\infty$.
\begin{Lem}[Lemma 10 of \cite{Col}]\label{lem10}
If $p\le3/4$ or $\ell_{\sup}(\alpha)<\frac{1}{4p-2}$, then $B_\infty^2=\infty$. Moreover, if $p>3/4$ and $\ell_{\inf}(\alpha)>\frac{1}{4p-2}$, then $B_\infty^2<\infty$.
\end{Lem}
If $p>3/4$ and $\ell_{\inf}(\alpha)>\frac{1}{4p-2}$, then $A_\infty<\infty$ by Lemma \ref{lem10}. We recall the following lemma without proof describing the non-degeneracy of the DERW differences.
\begin{Lem}[Lemma 12 of \cite{Col}]\label{lem12}Suppose one of the following three conditions:

\begin{enumerate}
\item $p<1$ and $\ell_{\inf}(\alpha)>0$.
\item  $p=1$ and $0<\ell_{\inf}(\alpha)\le\ell_{\sup}(\alpha)<1$, $0<\ell_{\inf}(\beta)\le\ell_{\sup}(\beta)<\frac{1-\ell_{\sup}(\alpha)}{1-\ell_{\inf}(\alpha)}$.
\item $0=\ell_{\inf}(\alpha)\le\ell_{\sup}(\alpha)<1$ and $0<\ell_{\inf}(\beta)\le\ell_{\sup}(\beta)<1-p\cdot\ell_{\sup}(\alpha)$.
\end{enumerate}
Then
\begin{align*}
\liminf_{n\to\infty}\mathop{\mathrm{Var}}[X_n]>0.
\end{align*}
\end{Lem}
Set $M_0:=0$ and
\begin{align*}
M_n:=\frac{S_n-E[S_n]}{a_n}\quad \text{for}\ n\ge1,\\
Y_n:=M_{n}-M_{n-1}\quad \text{for} \ n\ge1,
\end{align*}
By Proposition 11 of \cite{Col}, we see that $\{M_n\}_{n=1}^\infty$ is a martingale. Since $Y_n=M_n-E[M_n|\mathcal{F}_{n-1}]\ \text{a.s.}$, we have
\begin{align*}
Y_n=\frac{X_n-E[X_n|\mathcal{F}_{n-1}]}{a_n}\quad \text{a.s.}
\end{align*}
for $n\ge1$. By the definition of the DERW, we have
\begin{align}\label{eq: Y2}
|Y_n|\le\frac{2}{a_n}\quad \text{a.s.}
\end{align}
for $n\ge1$.
Let $s_n^2:=\sum_{k=n}^\infty E[Y_k^2]$ for $n\ge1$. If $p>3/4$ and $\ell_{\inf}(\alpha)>\frac{1}{4p-2}$, then $s_n^2\underset{n\to\infty}{\to}0$ by \eqref{eq: Y2} and Lemma \ref{lem10}.
\begin{Lem} \label{lem1}
Suppose one of the following two conditions:
\begin{enumerate}
\item $3/4<p<1$ and $\ell_{\inf}(\alpha)>\frac{1}{4p-2}$.
\item  $p=1$ and $\frac{1}{4p-2}<\ell_{\inf}(\alpha)\le\ell_{\sup}(\alpha)<1$, $0<\ell_{\inf}(\beta)\le\ell_{\sup}(\beta)<\frac{1-\ell_{\sup}(\alpha)}{1-\ell_{\inf}(\alpha)}$.
\end{enumerate}
Then $s_{n+1}^2\sim A_\infty^2-A_n^2$.
\end{Lem}
\begin{proof}
By Theorem 1 of \cite{Col}, we have
\begin{align*}
\frac{S_n-E[S_n]}{n}=o(1)\quad \text{a.s.}
\end{align*}
Thus, we have 
\begin{align*}
E[X_n|\mathcal{F}_{n-1}^X]-E[X_n]=(2p-1)\alpha_{n}\cdot\frac{S_{n-1}-E[S_{n-1}]}{n-1}=o(1)\quad \text{a.s.} 
\end{align*}
Put 
\begin{align*}
\xi_{n}:=E[X_n|\mathcal{F}_{n-1}^X]-E[X_n].
\end{align*} 
Then we have
\begin{align}
E[Y_n^2|\mathcal{F}_{n-1}^X]&=\frac{1}{a_n^2}E[1-2X_nE[X_n|\mathcal{F}_{n-1}^X]+E[X_n|\mathcal{F}_{n-1}^X]^2|\mathcal{F}_{n-1}^X]\notag\\
&=\frac{1}{a_n^2}(1-E[X_n|\mathcal{F}_{n-1}^X]^2)\notag\\
\label{eq: Y}
&=\frac{1}{a_n^2}(1-E[X_n]^2-2E[X_n]\xi_{n}-\xi_{n}^2).
\end{align}
Noting that $E[\xi_{n}]=0$, we have $E[Y_n^2]=\frac{1}{a_n^2}(1-E[X_n]^2-E[\xi_{n}^2])$. Then we have
\begin{align*}
\frac{s_{n+1}^2}{A_\infty^2-A_n^2}=1-\frac{1}{A_\infty^2-A_n^2}\sum_{k=n+1}^\infty\frac{1}{a_k^2}E[\xi_{k}^2].
\end{align*}
Note that $E[\xi_{n}^2]\to0$ as $n\to\infty$. In fact, since $\xi_{n}^2\le4$ for all $n\ge1$, we have $E[\xi_{n}^2]\to0$ as $n\to\infty$ by the Dominated Convergence Theorem. By Lemma \ref{lem12}, there exists $c\in(0,1)$ such that $1-E[X_n]^2\ge c$ for large $n$. Then
\begin{align*}
\frac{1}{A_\infty^2-A_n^2}\sum_{k=n+1}^\infty\frac{1}{a_k^2}E[\xi_{k}^2]&\le\frac{\sup_{k\ge n+1}E[\xi_{k}^2]}{A_\infty^2-A_n^2}(B_\infty^2-B_n^2)\\
&\le\frac{\sup_{k\ge n+1}E[\xi_{k}^2]}{c}\underset{n\to\infty}{\to}0. 
\end{align*}
The proof is complete.
\end{proof}
\begin{Lem} \label{lem2}
Suppose one of the following two conditions:
\begin{enumerate}
\item $3/4<p<1$ and $\ell_{\inf}(\alpha)>\frac{1}{4p-2}$.
\item  $p=1$ and $\frac{1}{4p-2}<\ell_{\inf}(\alpha)\le\ell_{\sup}(\alpha)<1$, $0<\ell_{\inf}(\beta)\le\ell_{\sup}(\beta)<\frac{1-\ell_{\sup}(\alpha)}{1-\ell_{\inf}(\alpha)}$.
\end{enumerate}
Then
\begin{align*}
\sum_{n=1}^\infty\frac{1}{s_n^4a_n^4}<\infty.
\end{align*}
\end{Lem}
\begin{proof}
Since $s_n^4\sim(A_\infty^2-A_{n-1}^2)^2$ holds by Lemma \ref{lem1}, it suffices to show $\sum_{n=1}^\infty1/(A_\infty^2-A_{n-1}^2)^2a_n^4<\infty$. By Lemma \ref{lem12}, there exists $c\in(0,1)$ such that $1-E[X_n]^2\ge c$ for large $n$. Thus we have for large $n$,
\begin{align*}
(A_\infty^2-A_{n-1}^2)^2a_n^4&=\left(\sum_{k=n}^\infty\frac{a_n^2}{a_k^2}(1-E[X_k]^2)\right)^2\\
&\ge c^2\left(1+\sum_{k=n+1}^\infty\frac{a_n^2}{a_k^2}\right)^2\\
&=c^2\left(1+\sum_{k=n+1}^\infty\prod_{l=n}^{k-1}\left(\frac{l}{l+(2p-1)\alpha_{l+1}}\right)^2\right)^2\\
&\ge c^2\left(1+\sum_{k=n+1}^\infty\prod_{l=n}^{k-1}\left(\frac{l}{l+1}\right)^2\right)^2\\
&=c^2\left(1+\sum_{k=n+1}^\infty\frac{n^2}{k^2}\right)^2\\
&\ge c^2\left(1+n^2\int_{n+1}^\infty\frac{1}{x^2}dx\right)^2\\
&=c^2\left(1+\frac{n^2}{n+1}\right)^2\\
&\ge c^2n^2.
\end{align*}
The proof is complete.
\end{proof}
\section{Proof of our main theorem}
We prove proceed to the proof of our theorem.
\begin{proof}[Proof of Theorem \ref{mThm}]
By Theorem \ref{nondegThm}, by Lemma \ref{lem12} and by Theorem 1 of \cite{Col}, the assumption of Theorem \ref{mThm} can be replaced by
\begin{align*}
B_\infty^2<\infty, \quad\liminf_{n\to\infty}\mathop{\mathrm{Var}}[X_n]>0, \quad p\cdot\ell_{\sup}(\alpha)<1.
\end{align*}
Let us prove Claim (i) (CLT). By Theorem 1 (b) of \cite{martconvthm1} and Corollary 1 of \cite{martconvthm1}, it suffices to check the following conditions:
\begin{enumerate}
\item[(a)]$\displaystyle
\frac{1}{s_n^2}\sum_{k=n}^\infty Y_k^2\to1$ as $n\to\infty$ in probability,
\item[(b)]$\displaystyle
\lim_{n\to\infty}\frac{1}{s_n^2}\sum_{k=n}^\infty E[Y_k^2;|Y_k|>\varepsilon s_n]=0$ for all $\varepsilon>0$.
\end{enumerate}

First we check Condition (a). Put
\begin{align*}
\eta_{n}:=-2E[X_n]\xi_{n}-\xi_{n}^2.
\end{align*} 
By \eqref{eq: Y}, we have $E[Y_n^2|\mathcal{F}_{n-1}^X]=\frac{1}{a_n^2}(1-E[X_n]^2+\eta_{n})$. Thus
\begin{align*}
\frac{\sum_{k=n}^\infty E[Y_k^2|\mathcal{F}_{k-1}^X]}{A_\infty^2-A_{n-1}^2}=1+\frac{1}{A_\infty^2-A_{n-1}^2}\sum_{k=n}^\infty\frac{\eta_{k}}{a_k^2}.
\end{align*}
By Lemma \ref{lem12}, there exists $c\in(0,1)$ such that $1-E[X_n]^2\ge c$ for large $n$. Noting that $\eta_{n}\to0$ as $n\to\infty$, we have
\begin{align*}
\left|\frac{1}{A_\infty^2-A_{n-1}^2}\sum_{k=n}^\infty\frac{\eta_{k}}{a_k^2}\right|\le\frac{\sup_{k\ge n}|\eta_{k}|}{c}\underset{n\to\infty}{\to}0.
\end{align*}
Hence $\frac{1}{s_n^2}\sum_{k=n}^\infty E[Y_k^2|\mathcal{F}_{k-1}]\to1$ as $n\to\infty$ by Lemma \ref{lem1}. We define 
\begin{align*}
L_n:=\sum_{k=1}^n\frac{1}{s_k^2}(Y_k^2-E[Y_k^2|\mathcal{F}_{k-1}^X])\quad n\ge1
\end{align*}
with $L_0:=0$, so that it is a martingale. In fact, $E[L_{n}-L_{n-1}|\mathcal{F}_{n-1}^X]=\frac{1}{s_{n}^2}(E[Y_{n}^2-E[Y_{n}^2|\mathcal{F}_{n-1}^X]|\mathcal{F}_{n-1}^X])=0$ for all $n\ge1$. By Lemma \ref{lem2} and by Eq. \eqref{eq: Y2}, we have
\begin{align*}
\sum_{k=1}^\infty \frac{1}{s_k^4}E[(Y_k^2-E[Y_k^2|\mathcal{F}_{k-1}^X])^2|\mathcal{F}_{k-1}^X]
&=\sum_{k=1}^\infty \frac{1}{s_k^4}(E[Y_k^4|\mathcal{F}_{k-1}^X]-E[Y_k^2|\mathcal{F}_{k-1}^X]^2)\\
&\le\sum_{k=1}^\infty\frac{1}{s_k^4}E[Y_k^4|\mathcal{F}_{k-1}^X]\\
&\le\sum_{k=1}^\infty\frac{16}{s_k^4a_k^4}<\infty.
\end{align*}
Thus $L_\infty<\infty,$\ a.s. by Theorem 2.15 of \cite{martconvthm2}, and so we obtain
\begin{align*}
\frac{1}{s_n^2}\sum_{k=n}^\infty (Y_k^2-E[Y_k^2|\mathcal{F}_{k-1}^X])\underset{n\to\infty}{\to}0
\end{align*}
by Lemma 1 (ii) of \cite{martconvthm1}. Then Condition (a) is satisfied.

Next we check Condition (b). For any $\varepsilon>0$, by Eq. \eqref{eq: Y2}, we have
\begin{align*}
\frac{1}{s_n^2}\sum_{k=n}^\infty E[Y_k^2;|Y_k|>\varepsilon s_n]\le\frac{1}{s_n^2}\sum_{k=n}^\infty E\left[Y_k^2\cdot\frac{Y_k^2}{\varepsilon^2s_n^2}\right]\le\frac{16}{\varepsilon^2s_n^4}\sum_{k=n}^\infty\frac{1}{a_k^4}\underset{n\to\infty}{\to}0
\end{align*}
by Lemma \ref{lem2} and Lemma 1 (ii) of \cite{martconvthm1}. The proof of Claim (i) (CLT) is now complete.

Let us prove Claim (ii) (LIL). By Theorem 1 (b) of \cite{martconvthm1}, it suffices to check the following conditions
\begin{enumerate}
\item[(a$'$)]$\displaystyle
\frac{1}{s_n^2}\sum_{k=n}^\infty Y_k^2\to1$ as $n\to\infty$\ a.s.,\\
\item[(c)]$\displaystyle
\sum_{k=1}^\infty\frac{1}{s_k}E[|Y_k|;|Y_k|>\varepsilon s_k]<\infty$ for all $\varepsilon>0$,\\
\item[(d)]$\displaystyle
\sum_{k=1}^\infty \frac{1}{s_k^4}E[Y_k^4;|Y_k|\le\delta s_k]<\infty$ for some $\delta>0$.
\end{enumerate}
By the proof of (i) of Theorem \ref{mThm}, Condition (a$'$) is satisfied. By Eq. \eqref{eq: Y2} and Lemma \ref{lem2}, we have
\begin{align*}
\sum_{k=1}^\infty\frac{1}{s_k}E[|Y_k|;|Y_k|>\varepsilon s_k]\le\sum_{k=1}^\infty \frac{1}{s_k}E\left[|Y_k|\cdot\frac{|Y_k|^3}{\varepsilon^3s_k^3}\right]\le\frac{16}{\varepsilon^3}\sum_{k=1}^\infty\frac{1}{s_k^4a_k^4}<\infty
\end{align*}
and
\begin{align*}
\sum_{k=1}^\infty \frac{1}{s_k^4}E[Y_k^4;|Y_k|\le s_k]\le\sum_{k=1}^\infty\frac{16}{s_k^4a_k^4}<\infty
\end{align*}
by Lemma \ref{lem2}. Then Conditions (c) and (d) are satisfied. The proof of Claim (ii) (LIL) is now complete.
\end{proof}\noindent
{\large\bf Acknowledgements}: The authors thank the anonymous referee for valuable comments to this paper, particularly for pointing out that Coletti et al. \cite{Col} reduces to our formulation. The authors would also like to thank Masato Takei and Naoki Kubota for their valuable comments. This research was supported by ISM. The research of K. Yano was supported by JSPS Open Partnership Joint Research Projects grant no. JPJSBP120209921 and by JSPS KAKENHI grant no.'s 19H01791, 19K21834 and 21H01002.

\end{document}